\documentclass{article}
\usepackage{amsfonts}
\usepackage{amsmath}
\usepackage{amssymb}
\usepackage{theorem}
\usepackage{graphicx}
\usepackage{epsfig}

\newtheorem{theorem}{Theorem}
\newtheorem{proposition}[theorem]{Proposition}
\newtheorem{lemma}[theorem]{Lemma}

\theorembodyfont{\rmfamily}

\newtheorem{remark}[theorem]{Remark}
\newcommand{\CAT}{{\mathrm{CAT}}}
\newcommand{\semi}{{:}}

\newcommand{\cc}{{\mathbf C}}
\newcommand{\zz}{{\mathbb Z}}
\newcommand{\redX}{{\overline X}}

\newcommand{\fp}{{\mathrm FP}}

\title{Every finite complex has the homology of some $\CAT(0)$ cubical duality group.}

\author{Raeyong Kim}

\date{\today}

\newenvironment{proof}[1][]{\begin{trivlist} \item[\hskip\labelsep
\emph{Proof#1.}]}{\foorp \end{trivlist}}
\newcommand{\foorp}{{\unskip\nobreak\hfil\penalty50
 \hskip1em\vadjust{}\nobreak\hfil \vrule height3pt width3pt depth0pt
 \parfillskip=0pt \finalhyphendemerits=0 \par}}

\begin{document} 

\maketitle

\begin{abstract} 
We prove that every finite connected simplicial complex has the homology of the classifying space for some $\mathrm{CAT}(0)$ cubical duality group. More specifically, for any finite simplicial complex $X$, we construct a locally $\CAT(0)$ cubical complex $T_{X}$ and an acyclic map $t_{X} : T_{X} \to X$ such that $\pi_{1}(T_{X})$ is a duality group.
\end{abstract} 


\section{Introduction} 

In \cite{KanThu}, Kan and Thurston proved that every simplicial complex $X$ has the homology of some aspherical space $T_{X}$. Equivalently, if $X$ is connected, then $X$ has the homology of the classifying space for some group $G$. The {K}an-{T}hurston theorem has been extended and generalized by a number of authors in many different ways. Baumslag, Dyer and Heller improved this result in \cite{BauDyeHel} that $T_{X}$ can be chosen to be a finite simplicial complex if $X$ is finite. Hausmann proved in \cite{Hau} that one can take $G$ to be a duality group if  $X$ is finite. Leary gave a new approach to the theorem by using metric geometry, namely, the classifying space for $G$ can be realized as a locally $\CAT(0)$ cubical complex.(See \cite{LeaMKT}.) The author also refers to \cite{McD}, \cite{Mau}, \cite{LeaNuc} and \cite{Kim} for more extensions and generalizations. In this paper, we prove the combination of Hausmann and Leary. In order to state the main theorem, we briefly recall the following definitions and introduce basic tools, which will be used in our construction.

A group $G$ is said to be of \emph{type $\fp$} if $\zz$, as a trivial $\zz G$-module, has a finite projective resolution over $\zz G$. A group $G$ of type $\fp$ is called a \emph{duality group of dimension $n$} if, for any abelian group $A$, $H^{i}(G \semi \zz G \otimes A) = 0$ for $i \neq n$. See \cite{Bro} and \cite{Bie} for equivalent formulations and examples. The following proposition says that duality groups behave well under amalgamations and HNN-extensions. It can be verified by the Mayer-Vietoris sequence. 

\begin{proposition}{\cite[Proposition 9.15]{Bie}}\label{dualamal}
Let $G$ be an amalgamated product $G=G_{1}*_{H}G_{2}$ or a HNN-extension $G=G_{1}*{H,\sigma}$. Assume that $G_1$ and $G_2$ are duality groups of dimension $n$ and that $H$ is a duality group of dimension $n-1$. Then, in either case, $G$ is a duality group of dimension $n$.
\end{proposition}

A geodesic metric space is said to be $\CAT(0)$ if any geodesic triangle is at least as thin as the comparison triangle in the Euclidean plane having the same side lengths. A geodesic space is called locally $\CAT(0)$ (or $nonpositively\,\, curved$) if every point has a $\CAT(0)$ neighborhood. See \cite{BriHae} for further details. The following gluing lemmas will be used throughout our construction.

\begin{lemma}{\cite[Proposition II.11.6, II.11.13]{BriHae}}\label{gluinglemma}
Let $X_1$ and $X_2$ be metric spaces of non-positive curvature.
\begin{enumerate}
\item (Simple Gluing Lemma) Let $A_i \subset X_i$ for $i=1,2$ be closed subspaces that are locally convex and complete. If $j: A_{1} \to A_{2}$ is a bijective local isometry, then the quotient of the disjoint union of $X_1$ and $X_2$ by the equivalence relation generated by $[a_{1} \sim j(a_{1}), \forall a_{1} \in A_{1}]$ is non-positively curved.
\item (Gluing with a Tube) Let $A$ be a metric space of non-positive curvature. If $A$ is compact and $\phi_{i} : A \to X_{i}$ is a local isometry for $i=1,2$, then the quotient of $\displaystyle{X_{1} \coprod (A \times [0,1]) \coprod X_{2}}$ by the equivalence relation generated by $[(a,0) \sim \phi_{1}(a), (a,1) \sim \phi_{2}(a), \forall a \in A]$ is non-positively curved.
\end{enumerate}
\end{lemma}

Our main theorem can be stated as follows.
\begin{theorem}\label{main}
For any finite simplicial complex $X$, there exists a finite locally $\CAT(0)$ cubical complex $T_{X}$ and a map $t_{X} : T_{X} \to X$ such that $\pi_{1}(T_{X})$ is a duality group and $t_{X}$ induces an isomorphism on homology for any local coefficients on $X$.
\end{theorem}

If $X$ is connected, $T_{X}$ is the classifying space for the group $\pi_{1}(T_{X})$. Therefore, Theorem \ref{main} says that every finite connected simplicial complex has the same homology as the classifying space for some $\mathrm{CAT}(0)$ duality group. 

Our proof follows the inductive strategy of \cite{Hau} directly, but we substitute for one of the groups used there  with one defined in \cite{LeaMKT}. We also use the proposition and gluing lemmas stated above to stay within the class of duality groups and the class of locally $\CAT(0)$ spaces. More specifically, in section \ref{twocomplexes}, we introduce two duality groups of dimension 2 and corresponding 2-complexes, which can be realized as locally CAT(0) cubical complexes. Based on these two locally CAT(0) cubical complexes, in section \ref{seqgroups}, we construct a sequence of locally CAT(0) acyclic cubical complexes and show that their fundamental groups are all duality groups. In section \ref{mainproof}, we use the previous construction to prove our main theorem.\\

This paper is a part of the author's Ph.D. thesis. The author would like to thank thesis advisor Ian Leary for his guidance throughout this research project. The author also would like to thank Jean Lafont for his careful reading of an earlier version of this paper.

\section{Two 2-complexes}
\label{twocomplexes}

Leary introduced the notion of tesselated $n$-gon made out of unit squares and applied Gromov's criterion to determine whether or not it is CAT(0). By using various tesselated CAT(0) $n$-gons, he gave some constructions for locally CAT(0) square complexes (See \cite[Section 2]{LeaMKT} for further details). The following group can be found in Proposition 3 in \cite{LeaMKT}. 

Let $H$ be the group presented on generators $a,b,c,d,e,f$ subject to the following 6-relators \semi
$$\begin{array}{ccc} abcdef,&ab^{-1}c^2f^{-1}e^{2}d^{-1},&a^2fc^2bed,\\ad^{-2}cb^{-2}ef^{-1},&ad^{2}cf^{2}eb^2,&af^{-2}cd^{-1}eb^{-2}\end{array}$$

Then the corresponding presentation 2-complex, denoted by $X_2$, can be realized as a non-positively curved square complex and it follows that $H$ is non-trivial and torsion-free. Also it can be easily verified that $H$ is acyclic. For the future use, we point out that $X_2$ can be constructed by one 0-cell, six loops and six 2-cells, where each loop consists of four 1-cells, and each loop is a closed geodesic.

\begin{lemma}\label{grouph}
$H$ is a duality group of dimension 2.
\end{lemma}

\begin{proof}
Let $A$ be an abelian group. Since $H$ is infinite, $H^{0}(H ; \zz H \otimes A) = 0$. 

For $H^{1}(H ; \zz H \otimes A)$,  we consider the link of vertex $lk(v,X_2)$ and the punctured link of vertex $Plk(v,X_2)$ in $X_2$. A direct observation verifies that $lk(v,X_2)$ and $Plk(v,X_2)$ are connected for any vertex $v$, so, by \cite{BraMei}, the universal cover $\widetilde{X_{2}}$ of $X_2$ is connected at infinity. It follows that $H$ is 1-ended.

Since $H^{1}(H ; \zz H \otimes A) \cong H^{1}(H ; \zz H)\otimes A$ and $H^{1}(H ; \zz H)$ is a free abelian group of rank $e-1$, where $e$ is the number of ends of the group, $H^{1}(H ; \zz H \otimes A)=0$.
\end{proof}

\begin{remark}
The result of \cite{BraMei} is much more general. Brady and Meier established sufficient conditions ensuring the higher connectivity at infinity and semistablity for locally finite CAT(0) cubical complexes. In particular, they described the higher connectivity at infinity of right angled Artin groups and determined which right angled Artin groups are duality groups.   
\end{remark}

We introduce the other group mentioned in the introduction. Let $K$ be the group presented on generators $v,w,x,y$ subject to a single relator $[v,w][x,y]y^{-1}$. As a one-relator group, $K$ is a duality group of dimension 2. (See \cite[Section 2]{Hau}.) 

\begin{lemma}\label{w}
The presentation 2-complex corresponding to $K$ may be realized as a non-positively curved square complex.
\end{lemma}

\begin{proof}
Consider the cubical subdivision of the tesselated nonagon appeared in Figure \ref{tesselatedgon}. Using this nonagon, one can construct the presentation $2$-complex with one $0$-cell, four loops and one $2$-cell, where each loop consists of four $1$-cells, and verify that each loop is a closed geodesic. Furthermore, by a direct calculation, all the vertex links are flag complexes, so, by Gromov's criterion, the presentation $2$-complex can be realized as a non-positively curved square complex.
\end{proof}

\begin{figure}
\setlength{\unitlength}{1cm}
\begin{center}
\begin{picture}(7,4)
\thicklines

\put(0.5,0.5){\circle*{.2}}
\put(2.5,0.5){\circle*{.2}}
\put(4.5,0.5){\circle*{.2}}
\put(6.5,0.5){\circle*{.2}}
\put(0.5,2.5){\circle*{.2}}
\put(1.5,3.5){\circle*{.2}}
\put(3.5,3.5){\circle*{.2}}
\put(5.5,3.5){\circle*{.2}}
\put(6.5,2.5){\circle*{.2}}

\put(0.5,0.5){\line(1,0){6}}
\put(0.5,1.5){\line(1,0){6}}
\put(0.5,2.5){\line(1,0){6}}
\put(1.5,3.5){\line(1,0){4}}

\put(2.5,0.5){\vector(-1,0){1.3}}
\put(4.5,0.5){\vector(-1,0){1.3}}
\put(4.5,0.5){\vector(1,0){1.3}}
\put(1.5,3.5){\vector(1,0){1.3}}
\put(3.5,3.5){\vector(1,0){1.3}}
\put(2.5,2.5){\vector(-1,0){1.3}}
\put(4.5,2.5){\vector(1,0){1.3}}

\put(0.5,2.5){\vector(0,-1){1.3}}
\put(6.5,0.5){\vector(0,1){1.3}}

\put(0.5,0.5){\line(0,1){2}}
\put(1.5,0.5){\line(0,1){3}}
\put(2.5,0.5){\line(0,1){3}}
\put(3.5,0.5){\line(0,1){3}}
\put(4.5,0.5){\line(0,1){3}}
\put(5.5,0.5){\line(0,1){3}}
\put(6.5,0.5){\line(0,1){2}}

\put(1.5,0){$v$}
\put(3.5,0){$w$}
\put(5.5,0){$x$}
\put(7,1.5){$y$}
\put(6,2.8){$x$}
\put(4.5,4){$y$}
\put(2.5,4){$y$}
\put(1,2.8){$v$}
\put(0,1.5){$w$}

\end{picture}
\end{center}
\caption{A tesselated $\CAT(0)$ nonagon with edge identifications.}
\label{tesselatedgon}
\end{figure}
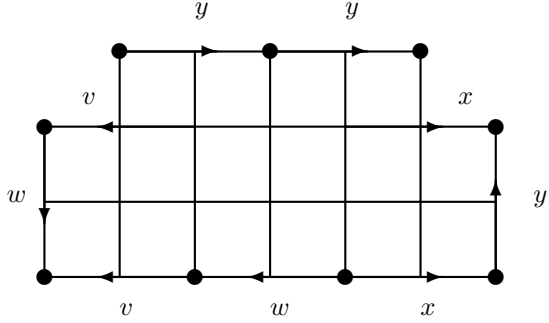

\section{Some acyclic complexes}
\label{seqgroups}
In this section, we construct a locally $\CAT(0)$ acyclic cubical complex $X_n$ for $n \geq 2$, whose fundamental group is a duality group. The following result is a version of \cite[Proposition 2.1]{Hau}, strengthened to work in the context of nonpositive curvature. The proof follows that in \cite{Hau}, but we use the group $H$ described above instead of $\Omega_{2}$ there and give explicit topological versions of its push-out constructions. Lemma \ref{gluinglemma} will control the geometry.

\begin{proposition}\label{acyclicgroups}
For $n \geq 2$, there is a locally $\CAT(0)$ cubical complex $X_n$ satisfying 
\begin{itemize}
\item $X_{n}$ is acyclic.
\item $\pi_{1}(X_{n})$ is a duality group of dimension $n$.
\item There is an isometric (and therefore $\pi_1$-injective) embedding $X_{n} \times S^{1} \to X_{n+1}$.
\end{itemize}
\end{proposition}

\begin{remark}
In this section and the next, we will repeatedly use the following homological facts, often without comments : 
\begin{enumerate}
\item The direct product of duality groups of dimension $m$ and $n$ is a duality group of dimension $m+n$. (\cite[Lemma 1.2]{Hau})
\item If $Y$ is acyclic, then $H_{*}(X\times Y) \simeq H_{*}(X)$. It follows from the K\"unneth formula.
\end{enumerate}
Similarly, we will use the geometric facts without further comments. Each is an easy exercise in $\CAT(0)$ geometry. We refer the reader to \cite{BriHae}.
\begin{enumerate}
\item If $f:C \to C'$ is an isometry between convex subcomplexes of cubical complexes $A$ and $B$, respectively, then 
$$A\cup_{f}B = A \coprod B/\{c \sim f(c) | c \in C\}$$
is a cubical complex in which $A$ and $B$ are isometrically embedded.
\item A subcomplex $Y$ of a cubical complex $X$ is isometrically embedded if and only if it is convex.
\end{enumerate}
\end{remark}

\begin{proof}[\,\,of Proposition \ref{acyclicgroups}]
Let $W$ be the presentation 2-complex corresponding to $K$, constructed in Lemma \ref{w}. From now on, we abuse symbols $a,b,\cdots, f$ and $v,\cdots, y$ for our purpose. They will either represent loops in $X_2$ and $W$, or group elements in $H$ and $K$, respectively.
Define

$$Y = W \mathop{\cup}_{x=f} X_{2},\quad Z_v = Y \mathop{\cup}_{w=f}  X_{2},\quad Z_w = Y \mathop{\cup}_{v=f} X_{2},$$
where we identify $0$-cells. Then one has
\begin{enumerate}
\item $Y, Z_{v}$ and $Z_{w}$ are locally $\CAT(0)$ cubical complexes. Each loop (of length 4) is closed, locally convex. So Lemma \ref{gluinglemma} applies.

\item $\pi_{1}(Y),\pi_{1}(Z_{v})$ and $\pi_{1}(Z_w)$ are duality groups of dimension 2. It follows from Lemma \ref{dualamal}.

\item By the Mayer-Vietoris sequence,

$H_1(Y\semi \zz) = \zz \oplus \zz$, generated by $v$ and $w$.

$H_1(Z_v\semi \zz) = \zz$, generated by $v$.

$H_1(Z_w\semi \zz) = \zz$, generated by $w$.

\item $H_2(Y\semi \zz) = H_2(Z_v\semi \zz) = H_2(Z_w\semi \zz) =0$.
\end{enumerate}

Let $S^{1}$ be a circle consisting of four 1-cells and four 0-cells. Consider $j_{c} : S^{1} \to X_{2}$ mapping $S^{1}$ simplicially onto the loop $c$, $j_{v} : S^{1} \to Z_{v}$ mapping $S^{1}$ simplicially onto the loop $v$ and $j_{w} : S^{1} \to Z_{w}$ is the map sending $S^{1}$ simplicially onto the loop $w$. Define 
$$U_{v} = (S^{1} \times Z_{v}) \coprod (X_{2} \times S^{1})/\sim,$$ 
where the equivalence relation is generated by $[(s,j_{v}(s')) \sim (j_{c}(s),s'), \forall s,s' \in S^{1}]$. Similarly, define 
$$U_{w}=(S^{1} \times Z_{w}) \coprod (X_{2} \times S^{1})/\sim,$$
where the equivalence relation is generated by $[(s,j_{w}(s')) \sim (j_{c}(s),s'), \forall s,s' \in S^{1}]$.

By Lemma \ref{gluinglemma}, $U_v$ and $U_w$ are locally $\CAT(0)$ cubical complexes. Since $\pi_{1}(S^{1} \times Z_{v}), \pi_{1}(S^{1} \times Z_{w})$ and $\pi_{1}(X_{2} \times S^{1})$ are duality groups of dimension 3 and $\pi_{1}(S^{1} \times S^{1})$ is a duality group of dimension 2, Lemma \ref{dualamal} implies that $\pi_{1}(U_v)$ and $\pi_{1}(U_w)$ are duality groups of dimension 3. Furthermore, by the K\"unneth formula and the Mayer-Vietoris sequence,
$$H_{1}(U_{v}) = \mathbb{Z}\,\, \textrm{(generated by $v$)},\,\,\, H_{2}(U_{v}) = 0$$
$$H_{1}(U_{w}) = \mathbb{Z}\,\, \textrm{(generated by $w$)},\,\,\, H_{2}(U_{w}) = 0$$

Using two inclusions $i_{v} : Y \hookrightarrow Z_{v} \hookrightarrow U_{v}$ and $i_{w} : Y \hookrightarrow Z_{w} \hookrightarrow U_{w}$, define $X_3$ to be

$$X_3 = U_{v} \mathop{\cup}_{Y} U_{w}$$

Then $X_{3}$ has the following properties :

\begin{enumerate}
\item $X_3$ is a locally $\CAT(0)$ cubical complex. It follows from Lemma \ref{gluinglemma}.
\item $X_3$ is acyclic and $\pi_{1}(X_3)$ is a duality group of dimension 3 : The Mayer-Vietoris sequence shows that $X_{3}$ is acyclic and $\pi_{1}(X_3)$ is a duality group of dimension 3 by Lemma \ref{dualamal}.
\item $X_{2} \times S^{1} \hookrightarrow U_{v} \hookrightarrow X_{3}$ is an isometric embedding and therefore, it is $\pi_{1}$-injective.
\end{enumerate}

Now suppose, by induction, that $X_i$ is constructed for $2 \leq i \leq n-1, n \geq 4$. Let $j_{f} : S^{1} \to X_2$ be the map sending $S^{1}$ onto the loop $f$. Using $j_{f}$ and the inclusion $i_{n-1} : X_{n-2} \to X_{n-1}$, define 
$$X_{n} = (X_{2} \times X_{n-2}) \coprod (S^{1} \times X_{n-1})/\sim , $$ 
where the equivalence relation is generated by $[(j_{f}(s), x) \sim (s,i_{n-1}(x)), \forall s \in S^{1},\, \forall x \in X_{n-2}]$.

Then $X_n$ satisfies all the property of the proposition.
\end{proof}


\section{Proof of the Theorem}
\label{mainproof}
In this section, we construct two functors to prove the main theorem.

For given a finite complex $X$, let $S(X)$ be the category whose objects are subcomplexes and morphisms are inclusions. For nonnegative integer $n$, let $\cc(n)$ be the category whose objects are compact locally $\CAT(0)$ cubical complexes such that the fundamental group of each component is a duality group of dimension $n$. Morphisms are isometric embeddings as closed subcomplexes.

\begin{theorem}
\label{main2}
Let $X$ be a finite complex. Then there exists $n$, depending on $X$, and two functors
\[L : S(X) \to \cc(n),\qquad M :S(X) \to \cc(n)\]
satisfying
\begin{enumerate}
\item For each connected subcomplex $Y$, there is an acyclic map $\beta_{Y} : L(Y) \to Y$.
\item For any subcomplex $Y$, $M(Y)$ is acyclic.
\item There is a natural transformation from $L$ to $M$, i.e. for any subcomplex $Y$, $L(Y)$ is isometrically embedded into $M(Y)$ and if $Y' \subset Y$, then the following diagram commutes.
$$\begin{array}{ccc}L(Y') & \to & M(Y')\\\downarrow&&\downarrow\\L(Y)& \to &M(Y)\end{array}$$
\end{enumerate}
\end{theorem}

\begin{remark} 
\begin{enumerate}
\item Theorem \ref{main2} is again a geometrized version of \cite[Proposition 3.1]{Hau}. For comparison, there $\cc(n)$ refers to the category of finite simplicial complexes, satisfying the same group-theoretic property, but no geometric one.
\item Theorem \ref{main} is a direct consequence of the above. (Take $L(X)$ for $T_{X}$.) 
\end{enumerate}
\end{remark}

\begin{proof}
The proof uses an induction on $m_{X}$, the number of simplices of $X$ of positive dimension.

Suppose that $X$ has no simplices of positive dimension, i.e. $m_{X}= 0$. In this case, $X$ is discrete. For any subset $Y$ of $X$, by taking $L(Y) = Y, M(Y) = CY$, the cone on $Y$, and $\beta_{Y}$ the identity, one can easily verify all the conditions in the Theorem.

Suppose that $(n,L,M,\beta)$ is constructed for $X$ and let $\redX = X \cup e$, where $e$ is a simplex of positive dimension. We then construct $(\overline{n},\overline{L},\overline{M},\overline{\beta})$ for $\overline{X}$.
\begin{enumerate}
\item For a subcomplex $Y$ of $\redX$, define

$$L_{1}(Y) = \left\{\begin{array}{cc}L(Y) \times X_3, & Y \subset X \\ (L(Y\cap X) \times X_3) \cup_{L(\partial e) \times X_2}  (M(\partial e) \times X_3) , & e \subset Y \end{array}\right.$$

If $Y \subset X$, then $L_1(Y)$ is a locally CAT(0) cubical complex as a product of two locally CAT(0) complexes. Suppose $e \subset Y$. In order to simplify formulations, we introduce the following notations.
$$L_1(\mathring{e}) = M(\partial e) \times X_3,\qquad \hat{L}_{1}(\partial e) = L(\partial e) \times X_2$$
Note that $L_1(\mathring{e})$ and $\hat{L}_{1}(\partial e)$ are locally $\CAT(0)$ cubical complexes and $\hat{L}_{1}(\partial e)$ is isometrically embedded in $L_1(\mathring{e})$ as a closed subcomplex. With notations above, if $e \subset Y$, then
$$L_1(Y) = (L(Y\cap X)\times X_3)\cup_{\hat{L}_{1}(\partial e)} L_1(\mathring{e}).$$
By induction, $\hat{L}_{1}(\partial e)$ is a closed, isometrically embedded subcomplex of $L(Y\cap X) \times X_3$ and $L_1(\mathring{e})$. Lemma \ref{gluinglemma} implies that  $L_1(Y)$ is a locally $\CAT(0)$ cubical complex. Also $\pi_{1}(L_{1}(Y))$ is a duality group of dimension $n+3$, as it is an amalgamated free product of two duality groups of dimension $n+3$ along a duality group of dimension $n+2$. This shows that $L_{1}(Y) \in \cc(n+3)$. 

At this step, $M_1$ will be defined only for $Y \subset X$ by

\[M_{1}(Y) = M(Y) \times X_3\]

\item  For $e \subset \redX$, define
$$\check{M}_{2}(e) = (\hat{L}_{1}(\partial e) \times X_3) \cup_{\hat{L}_{1}(\partial e) \times X_2} (L_{1}(\mathring{e})\times X_2)$$

Then $\check{M}_{2}(e) \in \cc(n+5)$. Again, we introduce the following notations.

\begin{enumerate}

\item $L_{2}(Y)  = L_{1}(Y) \times X_2$, for $Y\subset X$. Then $L_{2}(Y) \in \mathbf{C}(n+5)$

\item $M_{2}(Y) = M_{1}(Y) \times X_2$, for $Y \subset X$. Then $M_{2}(Y) \in \mathbf{C}(n+5)$

\item $\hat{L}_{2}(\partial e) = \hat{L}_{1}(\partial e) \times X_2 \in \mathbf{C}(n+4)$

\item $L_{2}(\mathring{e}) = L_{1}(\mathring{e})\times X_2 \in \mathbf{C}(n+5)$

\end{enumerate}

Using the notations introduced above, we have
$$\check{M}_{2}(e) = (\hat{L}_{1}(\partial e) \times X_{3}) \cup_{\hat{L}_{2}(\partial e)} L_{2}(\mathring e).$$

Note that
$$M_{2}(\partial e) = (M(\partial e)\times X_{3})\times X_{2} = L_{2}(\mathring{e}),\,\,\, \hat{L}_{2}(\partial e) = (L(\partial e)\times X_{2})\times X_{2}.$$
From the two inclusions $i_{0} : \hat{L}_{2}(\partial e) \hookrightarrow M_{2}(\partial e),\, i_{4} : \hat{L}_{2}(\partial e) \hookrightarrow L_{2}(\mathring e)$, consider the following double mapping cylinder.

$$K = M_{2}(\partial e) \mathop{\cup}_{i_0}\, (\hat{L}_{2}(\partial e) \times [0,4]) \mathop{\cup}_{i_4}  L_{2}(\mathring e)$$
By Lemma \ref{gluinglemma}, $K$ is locally $\CAT(0)$. We define the map $\phi$ from $K$ to $\check{M}_{2}(e)$ as follows.
\begin{itemize}
\item $\phi$ maps $M_{2}(\partial e)$ and $L_{2}(\mathring e)$ to $L_{2}(\mathring e) \subset \check{M}_{2}(e)$.
\item Recall in Proposition \ref{acyclicgroups} that $X_{2} \times S^{1}$ is isometrically embedded in $X_3$ and $S^{1}$ is represented by the loop $c$. $\phi$ maps the cylinder $\hat{L}_{2}(\partial e) \times [0,4] = \hat{L}_{1}(\partial e) \times X_{2} \times [0,4]$ to $\hat{L}_{1}(\partial e) \times X_{2} \times \{c\} \subset \hat{L}_{1}(\partial e) \times X_{3}$.
\end{itemize}

\begin{figure}
\begin{center}
\includegraphics[scale=0.3]{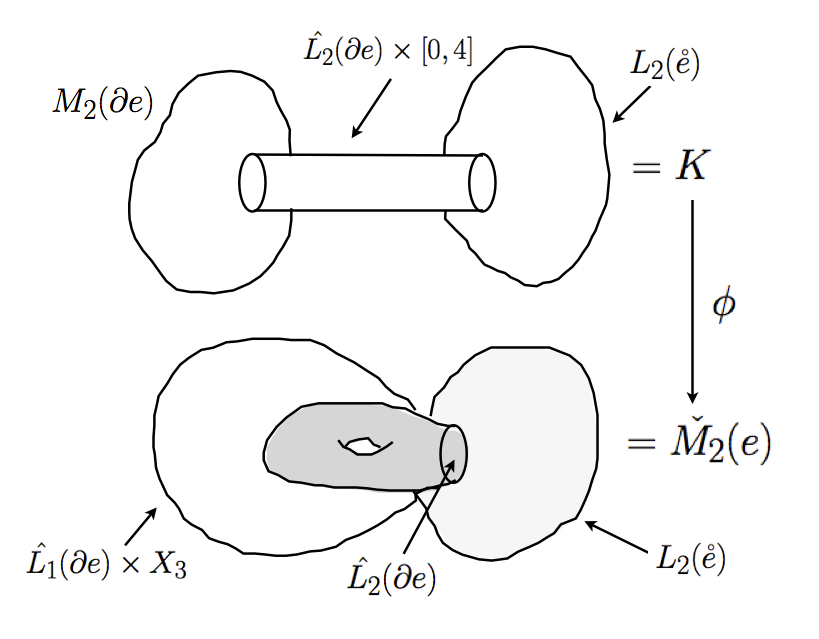}
\caption{The mapping Cylinder $M_{2}(e)$}
\end{center}
\label{mappingcylinder}
\end{figure}

Note that $K$ and $\check{M}_{2}(e)$ are compact and $\phi$ is a local isometry. Hence by Lemma \ref{gluinglemma}, the mapping cylinder by $\phi$ is also a locally $\CAT(0)$ cubical complex.  Denote this mapping cylinder by $M_{2}(e)$.(See Figure \ref{mappingcylinder}.) Since $M_{2}(e)$ is homotopy equivalent to  $\check{M}_{2}(e)$, $M_{2}(e) \in \cc(n+5)$.

\item For each subcomplex $Y$ of $\overline{X}$, define
$$M_{3}(Y) = \left\{\begin{array}{cc}M_{2}(Y) \times X_3,  & Y \subset X \\(M_{2}(Y\cap X) \times X_3) \cup_{M_{2}(\partial e)\times X_2} (M_{2}(e) \times X_3), & e \subset Y \end{array}\right.$$

Note that the inclusion of $M_{2}(\partial e)$ into $M_{2}(e)$ was explained above. 

Let $\check{L}_{2}(\mathring e) = \hat{L}_{2}(\partial e) \times [0,4] \mathop{\cup}_{i_{4}} L_{2}(\mathring e) \subset K$. For each $Y$ in $\redX$,

$$L_{3}(Y) = \left\{\begin{array}{cc}L_{2}(Y) \times X_2, & Y \subset X \\ (L_{2}(Y\cap X) \times X_2) \cup_{\hat{L}_{2}(\partial e)\times X_2}  (\check{L}_{2}(\mathring{e}) \times X_2), & e \subset Y \end{array}\right.$$

If $Y \subset X$, then it is obvious, by induction, that $L_{3}(Y)$ is isometrically embedded in $M_{3}(Y)$ as a closed subcomplex.

If $e \subset Y$, then $L_{3}(Y)$ is isometrically embedded in $M_{3}(Y)$ : Since every geodesic in $M_{3}(Y)$ connecting a point in $L_{2}(Y \cap X) \times X_2$ and a point in $\check{L}_{2}(\mathring{e}) \times X_2$ must pass through $\hat{L}_{2}(\partial e)\times X_2$, and $L_{2}(Y \cap X) \times X_2$ and $\check{L}_{2}(\mathring{e}) \times X_2$ are isometically embedded in $M_{2}(Y\cap X) \times X_3$ and $M_{2}(e) \times X_3$, respectively, therefore, it is a geodesic in $L_{3}(Y)$.  This verifies that $(L_{3}, M_{3})$ satisfies the condition 3.

\begin{figure}
\begin{center}
\includegraphics[scale=0.3]{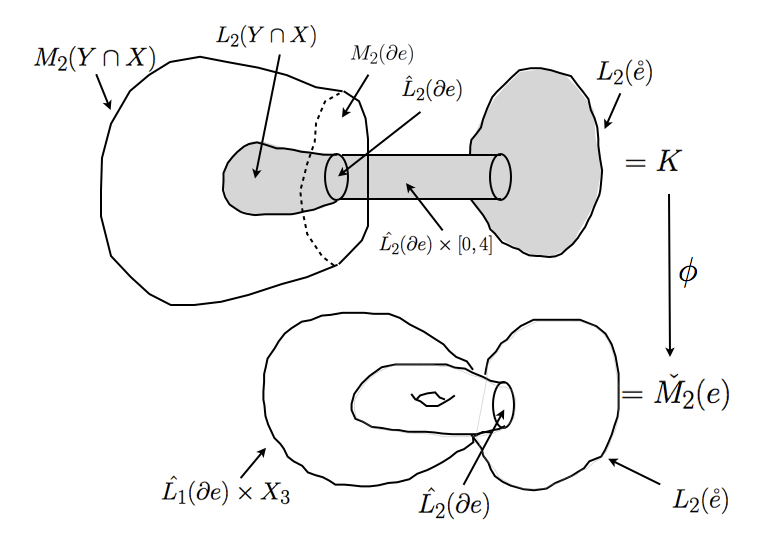}
\caption{$L_{3}(Y) \subset M_{3}(Y)$ if $e \subset Y$}
\end{center}
\label{L3inM3}
\end{figure}

Let $Y$ be a connected subcomplex of $\overline{X}$. 

If $Y \subset X$, then $L_{3}(Y) = L(Y) \times X_{3} \times X_{2} \times X_{2}$. By induction, there is an acyclic map from $L_{3}(Y)$ to $Y$. Suppose that $e \subset Y$. Recall $L_{1}(Y) = (L(Y\cap X) \times X_3) \cup_{L(\partial e) \times X_2}  (M(\partial e) \times X_3)$ and, by induction, $L(\partial e) = \partial e$ and $M(\partial e)$ is a cone on $\partial e$. By considering the maps $L(Y\cap X) \times X_{3} \to Y \cap X, L(\partial e) \times X_2 \to \partial e$ and $M(\partial e) \times X_3 \to e$, we obtain a map from $L_{1}(Y) \to Y$. The acyclicity follows from the Mayer-Vietoris sequence and the five lemma. By respecting products and using projections, this acyclic map extends to an acyclic map from $L_{3}(Y)$ to $Y$.(See Figure \ref{L3inM3}.) This verifies that $(L_{3},M_{3})$ satisfies condition 1.

Again, the Mayer-Vietoris sequence implies that $\check{M}_{2}(e)$, and therefore $M_{3}(Y)$ is acyclic.

Hence, $(L_3,M_3)$ satisfies all the conditions we want. But, note that $L_3$ is a functor to $\cc(n+7)$ and $M_{3}$ is a functor to $\cc(n+8)$.

In order for $\overline{L}$ and $\overline{M}$ to be functors to the same class $\cc(\overline{n})$, we modify  $L_{3}$ and $M_{3}$ and get $\overline{L}$ and $\overline{M}$ as follows.

For each $Y$ of $\redX$, define $\overline{L}(Y) = L_{3}(Y) \times X_3$.

Consider two inclusions $k_{1} : L_{3}(Y) \to M_{3}(Y)$ and $k_{2} : X_{2} \to X_{3}$. Define $\overline{M}$ to be the quotient of $(M_{3} \times X_{2}) \coprod \overline{L}(Y)$ by the equivalence relation generated by $[(k_{1}(y),x) \sim (y,k_{2}(x)) , \forall x \in X_{2}\,\forall y \in L_{3}(Y)]$.

Then $\overline{L}$ and $\overline{M}$ are functors to the same class $\cc(n+10)$ and satisfy all the conditions in the statement.

\end{enumerate}
\end{proof}

\begin{remark}
In the proof of Theorem \ref{main2}, $n$ gets increased by $10$ as a simplex of positive dimension is attached to $X$. In this way, one can give a very weak upper bound for the minimal dimension $n(X)$ of a duality group obtained from $X$. If $m(X)$ is the number of simplices of positive dimension, $n(X) \leq 10m(X)-7$. Note that $(L_{3},M_{3})$ is enough at the last stage. This upper bound was given in \cite{Hau} and we don't get any improvements in this paper.
\end{remark}

\bibliography{myref}{}
\bibliographystyle{plain}



\noindent 
 


\end{document}